\documentclass[11pt,a4paper,reqno]{amsart}
\usepackage{amscd,amssymb}
\usepackage[arrow,matrix]{xy}
\usepackage{xcolor}
\usepackage[colorlinks=true,linkcolor=blue,citecolor=purple]{hyperref}
\usepackage{graphicx}
\usepackage{fullpage}
\usepackage{mathrsfs}

\theoremstyle{plain}
\newtheorem{thm}[subsection]{Theorem}
\newtheorem{lem}[subsection]{Lemma}
\newtheorem{prop}[subsection]{Proposition}
\newtheorem{cor}[subsection]{Corollary}

\theoremstyle{definition}
\newtheorem{rk}[subsection]{Remark}
\newtheorem{definition}[subsection]{Definition}
\newtheorem{ex}[subsection]{Example}

\newcommand{\PP}{\mathbb{P}}
\newcommand{\ZZ}{\mathbb{Z}}

\newcommand{\MM}{\mathbb{M}}
\newcommand{\NN}{\mathbb{N}}
\newcommand{\Zc}{\mathcal{Z}}
\newcommand{\Jc}{\mathcal{J}}

\newcommand{\HF}{\mathrm{HF}}
\newcommand{\HP}{\mathrm{HP}}

\newcommand{\Proj}{\mathrm{Proj}}
\newcommand{\ann}{\mathrm{ann}}
\newcommand{\coker}{\mathrm{coker}\, }

\newcommand{\codim}{\mathrm{codim}}

\newcommand{\Bc}{\mathcal{B}}
\newcommand{\Sc}{\mathcal{S}}

\newcommand{\Rc}{\mathcal{R}}

\newcommand{\Kc}{\mathcal{K}}

\newcommand{\Sr}{\mathscr{S}}
\newcommand{\Ar}{\mathscr{A}}
\newcommand{\Rr}{\mathscr{R}}
\newcommand{\Cr}{\mathscr{C}}

\newcommand{\mm}{\mathfrak{m}}
\newcommand{\qq}{\mathfrak{q}}

\newcommand{\fb}{\mathbf{f}}
\newcommand{\xb}{\mathbf{x}}

\newcommand{\la}{\langle}
\newcommand{\ra}{\rangle}

\begin{document}

\title[Determinantal tensor product surfaces]{Determinantal tensor product surfaces and the method of moving quadrics}

\author{Laurent Bus\'e}
\address{Universit\'e C\^ote d'Azur, Inria, 2004 route des Lucioles, 06902 Sophia Antipolis, France.}
\email{laurent.buse@inria.fr}

\author{Falai Chen}
\address{School  of Mathematics, University of Science and Technology of China, Hefei, Anhui 230026,  China.}
\email{chenfl@ustc.edu.cn}


\date{\today}
\maketitle

\begin{abstract} A tensor product surface $\Sr$ is an algebraic surface that is defined as the closure of the image of a rational map $\phi$ from $\PP^1\times \PP^1$ to $\PP^3$. We provide new determinantal representations of $\Sr$ under the assumptions that $\phi$ is generically injective and its base points are finitely many and locally complete intersections. These determinantal representations are matrices that are built from the coefficients of linear relations (syzygies) and quadratic relations of the bihomogeneous polynomials defining $\phi$. Our approach relies on a formalization and generalization of the method of moving quadrics introduced and studied by David Cox and his co-authors.
\end{abstract}

\section{Introduction}\label{sec:intro}

The question of finding a determinantal representation of a projective hypersurface, i.e.~to write down its defining equation as the determinant of a square matrix filled with homogeneous forms, is a classical topic with a long history that goes back to the beginning of the last century (see for instance \cite{Beauville00} and the references therein). More recently, stimulated by applications in geometric modeling, the case of parameterized surfaces in $\PP^3$ has been extensively investigated, in particular for tensor product surfaces. An algebraic surface $\Sr$ in $\PP^3$ is called a \emph{tensor product surface} if it is the closure of the image of a rational map $\phi$ from $\PP^1\times \PP^1$ to $\PP^3$. Let $k$ be an algebraically closed field and let $R_1=k[s_0,s_1]$, $R_2=k[t_0,t_1]$ and $S=k[x_0,\ldots,x_3]$ be the coordinate rings of two $\PP^1$ and $\PP^3$. The rational map $\phi$ is of the form
$$\begin{array}{rlc}
\phi : \PP^1\times \PP^1 & \dasharrow & \Sr \subset \PP^3 \\
	(s_0:s_1)\times(t_0:t_1) & \mapsto & (f_0:f_1:f_2:f_3)
\end{array}
$$
where $f_0,f_1,f_2$ and $f_3$ are four bi-homogeneous polynomials in $R=R_1\otimes_k R_2$ of bidegree $(m,n)$, $m,n\geq 1$. Without loos in generality, one can assume that the base locus $\Bc$ of $\phi$ defines finitely many points.

\medskip

The classical method to obtain a determinantal representation of a tensor product surface is the Dixon resultant  which dated back to 1909 \cite{DixonRes}. However, this method only applies if $\phi$ has no base points and it builds a relatively large matrix of size $2mn$. At the end of the 20th century, the problem of determining a determinantal representation of a parameterized surface, also called the implicitization problem, received a lot of interest motivated by applications in geometric modeling for computing intersections between algebraic curves and surfaces by means of computer-assisted methods. Thus, in order to obtain compact determinantal representations and to overcome the difficulty of base points, a new empirical technique has been introduced by Sederberg and Chen and called \emph{the method of moving surfaces}  \cite{SC95}. They verified dozens of examples to demonstrate the validity of this method but without any rigorous proof. A few years later, David Cox noticed that syzygies of the polynomials $\fb:=(f_0,\ldots,f_3)$ and their quadratic relations are at the center of this technique and then proving its validity became an attractive open problem and a source of further developments in other connected topics (see e.g.~\cite[Chapter 1 and Chapter 3]{CoxCBMS} and the references therein).

\medskip

In this paper we will focus on the method of moving surfaces over $\PP^1\times \PP^1$ for tensor product surfaces. In this setting, this method is called the method of moving quadrics and it can be described as follows. A \emph{moving plane} of bidegree $(\mu,\nu)\in \NN^2$ is a polynomial of the form
$$L(s_0,s_1;t_0,t_1;x_0,x_1,x_2,x_3)=\sum_{i=0}^3 g_i(s_0,s_1;t_0,t_1)x_i$$
where $g_0,g_1,g_2,g_3$ are bihomogeneous polynomials in $R$ of bidegree $(\mu,\nu)$. This polynomial defines a family of planes in $\PP^3$ parameterized by $\PP^1\times\PP^1$, hence its name. In addition, the moving plane $L$ is said to follow $\phi$ if $\sum_i f_ig_i=0$ in $R$. Geometrically, this means that the plane of equation $L=0$ goes through the point $\phi(s_0,s_1;t_0,t_1) \in \Sr$. The moving planes of bidegree $(\mu,\nu)$ following $\phi$ form a $k$-vector space that we denote by $V_{(\mu,\nu)}$. Algebraically, $V_{(\mu,\nu)}$ is simply the vector space of syzygies of $\fb$ of bidegree $(\mu,\nu)$.

Similarly, a \emph{moving quadric} of bidegree $(\mu,\nu)$ is a polynomial of the form  
$$Q(s_0,s_1;t_0,t_1;x_0,x_1,x_2,x_3)=\sum_{0\leq i \leq j\leq 3} g_{i,j}(s_0,s_1;t_0,t_1)x_ix_j$$
where the $g_{i,j}$'s are bihomogeneous polynomials in $R$ of bidegree  $(\mu,\nu)$. It is said to follow $\phi$ if $\sum_{i,j}g_{i,j}f_if_j=0$ and we denote by $W_{(\mu,\nu)}$ the $k$-vector space of moving quadrics following $\phi$ of bidegree $(\mu,\nu)$, equivalently the space of quadratic relations of $\fb$ of bidegree $(\mu,\nu)$ in $R$. By multiplying a moving plane in $V_{(\mu,\nu)}$ with a linear form in $S_1$ we obtain a (degenerated) moving quadric in $W_{(\mu,\nu)}$; we denote by $V'_{(\mu,\nu)}$ the vector subspace of $W_{(\mu,\nu)}$ generated this way by moving planes in $V_{(\mu,\nu)}$.   

Now, choose a couple of integers $(\mu,\nu)$, a basis $\langle L_1,\ldots,L_l \rangle$ of $V_{(\mu,\nu)}$, a basis $\langle Q_1,\ldots,Q_q \rangle$ of $W_{(\mu,\nu)}/V'_{(\mu,\nu)}$ and a basis $\mathfrak{B}=\langle b_1,\ldots,b_{(\mu+1)(\nu+1)}\rangle$ of the vector space $R_{(\mu,\nu)}$ of bihomogeneous polynomials of bidegree $(\mu,\nu)$ in $R$. The matrix $\MM_{(\mu,\nu)}$ built from the coefficients of $L_i$ and $Q_j$ with respect to $\mathfrak{B}$, i.e.~the matrix which satisfies to the equality
$$ \left[ b_1 \, \cdots \, b_{(\mu+1)(\nu+1)} \right] \cdot \MM_{(\mu,\nu)} = \left[ L_1 \, \cdots \, L_l \ Q_1 \, \cdots \, Q_q  \right],$$
is called the matrix of moving planes and quadrics of bidegree $(\mu,\nu)$ following $\phi$. The method of moving quadrics then asks if there exists a couple of integers $(\mu,\nu)$ such that the matrix $\MM_{(\mu,\nu)}$ is square and its determinant yields an implicit equation of $\Sr$. Geometrically, this means that the linear system composed of these $l$ moving planes and $q$ moving quadrics cut out exactly the surface $\Sr$. 

\medskip

 The first theoretical result on the validity of the method of moving quadrics for tensor product surfaces is proved in the seminal paper \cite{CGZ00}: Assume that $\phi$ has no base point and is generically injective, if there are no moving planes of bidegree $(m-1,n-1)$ following $\phi$, i.e.~if $V_{(m-1,n-1)}=0$, then $W_{(m-1,n-1)}$ is of dimension $mn$ and $\MM_{(m-1,n-1)}$ yields a determinantal representation of $\Sr$ built solely with moving quadrics. Since then, the question of whether this result can be generalized and can adapt automatically to the presence of base points  became a research topic of interest (see the comments and open questions in \cite[Section 6]{CGZ00} and \cite[Section 6]{Cox01}). To the best of our knowledge, the more general result on the validity of the method of moving quadrics in the presence of base points is given in \cite{AHW05} where the authors show that $\MM_{(m-1,n-1)}$ yields a determinantal representation of the surface under a list of several restrictive hypothesis (see \cite[Section 4]{AHW05}). 
 
In this paper we prove the validity of the method of moving quadrics under mild assumptions and hence obtain new determinantal representations for tensor product  surfaces in the presence of base points. To achieve this goal, we provide a new interpretation of this method by lifting $\phi$ to a certain blowup of $\PP^1\times \PP^1$ along its base locus $\Bc$. From a more algebraic perspective, we interpret matrices $\MM_{(\mu,\nu)}$ as presentation matrices of some graded components of the symmetric algebra of the ideal $I=(\fb)$ mod out by its torsion, which is nothing but the Rees algebra of $I$ if the base points are all assumed to be locally complete intersections. As we will see, a key ingredient in our approach to adapt the presence of base points is the maximum degree $\mu_0$ of a minimal syzygy of the $R_1$-module of syzygies of $I$ of degree $n-1$ with respect to the grading induced by $R_2$. We define $\nu_0$ similarly by considering the syzygies of $\fb$ of degree $m-1$ with respect to the grading induced by $R_1$. We will prove the following result:

\begin{thm}\label{thm:intro1} Assume that $\phi$ is generically injective and that its base points, if any, are locally complete intersections. 
Then, the matrix $\MM_{(\mu_0-1,n-1)}$ (resp.~$\MM_{(m-1,\nu_0-1)}$) yields a determinantal representation of $\Sr$ if and only if the ideal $\fb'=(f_0',f_1',f_2')$ generated by three general $k$-linear combinations of $\fb$ has no syzygy in bidegree $(\mu_0-1,n-1)$ (resp.~$(m-1,\nu_0-1)$).
\end{thm}

This theorem generalizes the previously known results on the validity of the method of moving quadrics (e.g.~\cite{CGZ00,AHW05,CZS01}). It also covers the more recent results proved in \cite{LAI20191} where this method has been revisited in the absence of base points in order to remove the assumption $V_{(m-1,n-1)}=0$ used in \cite{CGZ00}. Actually, the new algebraic interpretation we propose of the method of moving quadrics allows us to prove the following general result.

\begin{thm}\label{thm:intro2} Assume that $\phi$ is generically injective and that its base points, if any, are locally complete intersections. Then, for any  $\mu\geq \mu_0$ there exists an exact sequence of graded free $S$-modules 
	$$ 0 \rightarrow S(-2)^r \xrightarrow{\partial_2} S(-1)^l \oplus S(-2)^q \xrightarrow{\partial_1} S^{\mu n}$$
whose determinant yields an implicit equation of the tensor product surface $\Sr$. The matrix of $\partial_1$ is $\MM_{(\mu-1,n-1)}$ and the matrix of $\partial_2$ is filled with linear forms obtained from second order Koszul syzygies of $\fb$. Moreover, this complex is a free resolution of the graded component of bidegree $(\mu-1,n-1)$ of the Rees algebra of the ideal $I$. 

Analogous results hold by symmetry for any $\nu\geq \nu_0$ using the matrices $\MM_{(m-1,\nu-1)}$.
\end{thm}

\noindent As we will see, when the map $\phi$ is not generically injective but only generically finite onto $\Sr$, both Theorem \ref{thm:intro1} and Theorem \ref{thm:quadres} still hold but the determinants yield an implicit equation of $\Sr$ raised to the power the degree of $\phi$. The presence of base points that are locally almost complete intersection, i.e.~that are locally defined by at most 3 generators, will also be considered. 

Finally, for the sake of completeness we mention that a variant of the method of moving quadrics for triangular surfaces, i.e.~for surfaces that are obtained as the closure of the image of a rational map from $\PP^2$ to $\PP^3$, has also been introduced in \cite{CGZ00} in the absence of base points. Then, it has been  further studied in \cite{BCD03} and \cite{BCS10} in the presence of base points. In particular, a result similar to Theorem \ref{thm:intro2} above is proved in \cite[Section 5]{BCS10} where a threshold integer $\mu_0$ is also defined for triangular surfaces (see \cite[Definition 1]{BCS10}). However, this threshold integer is much less explicit than the one we defined above for tensor-product surfaces and to the best of our knowledge, there is no result on  determinantal representations for triangular surfaces that is similar to Theorem \ref{thm:intro1} above.

\medskip

The paper is organized as follows. Section \ref{sec:prem} is devoted to some preliminaries on determinants and blowing-up algebras, more precisely the symmetric and the Rees algebras of certain bihomogeneous ideals in $R$. In Section \ref{sec:locSI} we analyze the local cohomology modules of the symmetric algebra $\Sc_I$ (Proposition \ref{prop:HmSI}) and as a first consequence we recover results by Botbol \cite{Bot11} about representations of $\Sr$ by means of matrices built solely from linear forms, in particular from moving planes following $\phi$ (Corollary \ref{cor:Bot}).  In Section \ref{sec:MQ-SI} we prove Theorem \ref{thm:botbol}, i.e.~that tensor product surfaces can always be represented as the ratio of two determinants, the one in the numerator being a maximal minor of a matrix $\MM_{(\mu,\nu)}$ of moving planes and quadrics following $\phi$ (Theorem \ref{thm:main}). Finally, in order to prove Theorem \ref{thm:intro1},  we investigate in Section \ref{sec:detformula} under which condition this ratio of determinants can be reduced to a single determinant and hence yields a determinantal representation (Theorem \ref{thm:quadres}). We close the paper by explaining how our results cover the previously known determinantal representations of tensor product surfaces and by providing illustrative examples.

\section{Preliminaries on blowing-up algebras}\label{sec:prem}

As already mentioned in the previous section, we suppose that a rational map $\phi$ is given:
$$\begin{array}{rlc}
\phi : \PP^1\times \PP^1 & \dasharrow & \Sr \subset \PP^3 \\
	(s_0:s_1)\times(t_0:t_1) & \mapsto & (f_0:f_1:f_2:f_3),
\end{array}
$$
where $f_0,f_1,f_2$ and $f_3$ are four bihomogeneous polynomials in $R=R_1\otimes_k R_2$ of bidegree $(m,n)$, $m,n\geq 1$.  Without loos in generality, one can assume that the base locus of $\phi$ defines finitely many points. We also assume that the closure of the image of $\phi$ is a surface $\Sr\subset \PP^3$. The base scheme of $\phi$ is the scheme $\Bc=V(I) \subset \PP^1\times \PP^1$ where we recall that $I$ is the bigraded ideal of $R$ generated by $\fb$. For each base point $p\in \Bc$ we denote by $d_p$ its multiplicity and by $e_p$ its Hilbert-Samuel multiplicity; it is well known  that $e_p\geq d_p$ and that equality holds if and only if $p$ is locally a complete intersection (abbreviated l.c.i.), i.e.~locally defined by 2 generators. Finally, we recall that (see  \cite[Proposition 4.4]{Fulton} and \cite[Appendix]{Cox01})
\begin{equation}\label{eq:degformula}
\deg(\phi)\deg(\Sr)=2mn-\sum_{p\in \Bc} e_p	
\end{equation}
where $\deg(\phi)$ stands for the degree of the generically finite map $\phi:\PP^1\times\PP^1 \dasharrow \Sr \subset \PP^3$.

\subsection{Determinants} Let $\MM$ be a matrix of an injective graded map $\rho$ between two free $S$-modules of the same rank $$0 \rightarrow F_1:=\oplus_{i=1}^lS(-e_i) \xrightarrow{\rho} F_0:=\oplus_{i=1}^l S(-d_i).$$
If $\det(\MM)$ is a defining equation of $\Sr$ then the $S$-module $Q:=\coker(\rho)$ has a finite free resolution of length one and is supported on $\Sr$, i.e.~$\ann_S(Q)=(F)$ where $F=0$ is a defining equation of $\Sc$. Conversely, the existence of such $S$-modules yields determinantal representations of $\Sr$, possibly raised to some power (see e.g.~\cite[\S 1]{Beauville00}, \cite{KM76} or \cite{North} where these modules are called elementary modules).

More generally, suppose given a graded $S$-module $M$ and a finite free resolution $F_\bullet$ of length $p\geq 1$:
$$ 0\rightarrow F_p \rightarrow F_{p-1} \rightarrow \cdots \rightarrow F_1 \rightarrow F_0 \rightarrow M \rightarrow 0.$$
If $\ann_S(M)\neq 0$, or equivalently if the Euler characteristic of $F_\bullet$ is equal to zero, then the complex $F_\bullet$ admits a determinant $\det(F_\bullet)$ which is supported on the codimension one components of $\ann_S(M)$ (see \cite{KM76} and \cite[Appendix A]{GKZ}; see also \cite{North} where these determinants are called MacRae's invariants). More precisely,
$$[\det(F_\bullet)]=\mathrm{div}_S(M)=\sum_{\stackrel{\qq \, \mathrm{prime},}{\codim_S(\qq)=1}}  \mathrm{length}_{S_\qq}(M_\qq)[\qq].$$

Thus, to obtain representations of the surface $\Sr$ as an alternate product of determinants of matrices, one may seek for free resolutions of graded $S$-modules that are supported on $\Sr$. Of course, free resolutions of length one are of particular interest as they yield determinantal representations. In this paper, we will consider modules that are built from the blowing-up of $\PP^1\times \PP^1$ along the base scheme $\Bc$ of the map $\phi$.

\subsection{The Rees algebra} Let $\Gamma \subset \PP^1\times \PP^1 \times \PP^3$ be the Zariski closure of the graph of the map $\phi : \PP^1\times \PP^1 \setminus \Bc \rightarrow \PP^3$. It is well known that $\pi(\Gamma)=\Sr$ where $\pi$ denotes the canonical projection of $\PP^1\times \PP^1 \times \PP^3$ on its third factor $\PP^3$, which is a projective morphism. Algebraically, one has
$$\Gamma=\Proj(\Rc_I)$$
where $\Rc_I:=R\oplus I \oplus I^2 \oplus \cdots=\oplus_{n\geq 0} I^n$ is the Rees algebra of the ideal $I$ in $R$ \cite[II.7]{Hart}. The embedding $\Gamma \subset \PP^1\times \PP^1 \times \PP^3$ corresponds to the graded map
\begin{eqnarray*}\label{eq:reesdef}
 R \otimes_k S= R[x_0,x_1,x_2,x_3] &  \xrightarrow{\beta} & \Rc_I \\
 x_i & \mapsto & f_i \in I^1
\end{eqnarray*}
and the homogeneous polynomials in $\Kc=\ker(\beta)$ are the defining equations of $\Rc_I$. It is important to notice that $\Rc_I$ is, as $\Gamma$, trigraded. For instance, the graded component of $\Kc$ of degree 1 with respect to the grading induced by $S$ is in correspondence with the module of syzygies of the polynomials $f_0,\ldots,f_3$, i.e.~the moving planes:
\begin{equation}\label{eq:K1}
g_0x_0+g_1x_1+g_2x_2+g_3x_3 \in \Kc_1 \Leftrightarrow g_0f_0+g_1f_1+g_2f_2+g_3f_3=0.
\end{equation}
Similarly, the graded component of $\Kc$ of degree 2 with respect to this grading correspond to the moving quadrics.

On the other hand, for any couples of integers $(\mu,\nu)$ let us denote by ${\Rc_I}_{(\mu,\nu)}$ the graded component of $\Rc_I$ with respect to its natural bigrading induced by $R=R_1\otimes_k R_2$. All these graded components are $S$-modules that are supported on $\Sr$; one has $\ann_S({\Rc_I}_{(\mu,\nu)})=(F)$ for any couples $(\mu,\nu)$ of non negative integers, where $F=0$ is a defining equation of $\Sr$, since $\Rc_I$ is a domain (see also \cite[Lemma 4.4]{Bot11}). As an extreme case, observe that ${\Rc_I}_{(0,0)}$ is isomorphic to $S/(F)$ and that it yields a $1\times 1$-matrix $\MM$ whose single entry is equal to $F$.

\subsection{Moving planes and the symmetric algebra}\label{subsec:MPSI}
Another interesting blowing-up algebra attached to the ideal $I$ is its symmetric algebra $\Sc_I$. It is simpler than the Rees algebra as it is defined by the moving planes following $\phi$. More precisely, $\Sc_I\simeq R[x_0,\ldots,x_3]/\Jc\la 1\ra$ where $\Jc\la 1\ra$ is the ideal generated by elements in $\Kc_1$ \eqref{eq:K1}.
The canonical surjective map $\Sc_I \rightarrow \Rc_I$ corresponds to the inclusion
$$\Gamma \subset \Gamma':=\Proj(\Sc_I)\subset \PP^1\times \PP^1\times \PP^3.$$
If all the base points are l.c.i.~then $\Gamma=\Gamma'$ and hence $\pi(\Gamma')=\Sr$ (see for instance \cite[Theorem 5.5 and Lemma 5.6]{Bot11}). Algebraically, this equality means that the prime ideal $\Kc$ is the saturation of $\Jc\la 1\ra$ with respect to the product of ideals $\mm=\mm_1\mm_2$, where $\mm_1=(s_0,s_1)$ and $\mm_2=(t_0,t_1)$, i.e.~$\Kc = \Jc\la 1 \ra : \mm^\infty$. We notice that if some base points are almost local complete intersection (abbreviated a.l.c.i.), i.e.~locally defined by 3 generators, then $\Gamma'$ has some extra components with respect to $\Gamma$: for each a.l.c.i.~base point $p$ a linear form $L_p \in S$ appears in $\Gamma'$ with multiplicity $e_p-d_p$ (see \cite{BCJ09} and \cite[Lemma 5.6]{Bot11}).

Similarly to the Rees algebra $\Rc_I$, the symmetric algebra $\Sc_I$ is also trigraded and its graded components ${\Sc_I}_{(\mu,\nu)}$ with respect to the bigrading of $R$ yields graded $S$-modules. Assuming that all the base points are l.c.i.~these modules are supported on $\Sr$ providing the graded components of $\Kc$ and $\Jc\langle 1\rangle$ are equal in degree $(\mu,\nu)$, that is to say providing $H^0_\mm(\Sc_I)_{(\mu,\nu)}=0$.

The main reason why we are  interested in the symmetric algebra $\Sc_I$ is because graded free resolutions of ${\Sc_I}_{(\mu,\nu)}$ can be built under suitable assumptions.

\subsection{The approximation complex of cycles} Consider the two sequences $\fb=(f_0,\ldots,f_3)$ and $\xb=(x_0,\ldots,x_3)$ of 
elements in $T=R\otimes_k S= R[\xb]$ and their corresponding homological Koszul complexes $K_\bullet(\fb;T)$ and $K_\bullet(\xb;T)$. They both have the same modules $K_p\simeq \wedge^p(T^4)$. We denote by $Z_p(\fb;T)$ and $Z_p(\xb;T)$ their cycles and by $d_\bullet^\fb$ and $d_\bullet^\xb$ their differentials, respectively. One has $d_{p-1}^{\fb}\circ d_{p}^\xb+d_{p-1}^\xb\circ d_{p}^\fb=0$ and hence for any $p$
$$ d_p^\xb ( Z_p(\fb;T)) \subset Z_{p-1}(\fb;T).$$
The complex $\Zc_\bullet$, whose modules are $Z_p(\fb;T)$ and differentials are $d_\bullet^\xb$, is called the approximation complex of cycles associated to $\fb$. It has been introduced in \cite{HSV82,HSV83} to study properties of Rees and symmetric algebras, in particular their Cohen-Macaulayness. More recently, these complexes have been used for analyzing images of rational maps (see e.g. \cite{BuJo03,Bot11,BBC14}). Among their numerous properties, for our purpose it is important to emphasize the following ones. First $Z_p(\fb;T)=Z_p(\fb;R)\otimes T$, which means that the modules of $\Zc_\bullet$ are essentially the cycles of the Koszul complex $K_\bullet(\fb;R)$. In addition, $Z_0(\fb;R)=R$ and $Z_1(\fb;R)$ is the first module of syzygies of $\fb$, therefore $H_0(\Zc_\bullet)\simeq \Sc_I$. The homology modules $H_i(\Zc_I)$ are $\Sc_I$-modules and they only depend on the ideal $I \subset R$ up to isomorphism.

The complex $\Zc_\bullet$ is naturally trigraded: a bigrading is induced by $R$ and a standard one by $S$. In order to have the differentials $d_\bullet^\xb$ of degree 0, the modules $\Zc_p$ are graded by
\begin{equation}\label{eq:Zc}
\Zc_p\simeq Z_p[(pm,pn)] \otimes S(-p)	
\end{equation}
where the shifts in gradings $[-]$ and $(-)$ are with respect to the grading induced by $R$ and $S$, respectively. We notice that the modules $Z_p$ are canonically graded by the standard grading of the (homological) Koszul complex $K_\bullet(\fb;T)$. Thus, the graded piece of bidegree $(\mu,\nu)$ (with respect to the grading of $R$) of $\Zc_\bullet$ is a graded complex of free $S$-modules of the form
$${\Zc_\bullet}_{(\mu,\nu)} \ : \ \cdots \rightarrow {Z_2}_{(\mu,\nu)+2(m,n)}\otimes_k S(-2) \xrightarrow{d_2^\xb}
 {Z_1}_{(\mu,\nu)+(m,n)}\otimes_k S(-1) \xrightarrow{d_1^\xb} R_{(\mu,\nu)}\otimes_k S.$$
The maps in this complex are given by matrices whose entries are linear forms in $S$. More specifically, the columns of a matrix of $d_1^\xb$ are filled with moving planes of $\phi$ of bidegree $(\mu,\nu)$. Moreover, under appropriate assumptions, its determinant yields a defining equation of $\Sr$.

\begin{thm}[\cite{Bot11}]\label{thm:botbol} Let $\mathscr{R}$ be the following subset of $\ZZ^2$
	\begin{multline}\label{eq:areaR}
	\mathscr{R}:=\{ (\mu,\nu)\in \ZZ^2 : \mu\leq 2m-2, \nu\leq 2n-2 \}\cup 	 \\
	\{ (\mu,\nu)\in \ZZ^2 : \mu\geq m,  \nu\leq n-2 \}\cup 	
	\{ (\mu,\nu)\in \ZZ^2 : \mu\leq m-2,  \nu\geq n \}.
	\end{multline}
	Then, for any $(\mu,\nu)\notin \Rc$ the complex $(\Zc_\bullet)_{(\mu,\nu)}$ is acyclic. Moreover, if all the base points of $\phi$ are l.c.i.~then the determinant of $(\Zc_\bullet)_{(\mu,\nu)}$ is equal to $F^{\deg(\phi)}$ where $F=0$ is a defining equation of $\Sr$.
\end{thm}
We notice that the subset $\Rc$ corresponds to the union of light and dark grey areas illustrated in Figure \ref{fig:domains}. Theorem \ref{thm:botbol} rarely provides determinantal formula \cite[Lemma 7.3]{Bot11}. A case where this happens is when the map $\phi$ has no base point. Under this assumption the complex $(\Zc_\bullet)_{(2m-1,n-1)}$ has length one and it recovers the Dixon's $(2mn\times 2mn)$-matrix of linear forms in $S$ (see also \cite{CGZ00}). We will recover this result as well as a proof of Theorem \ref{thm:botbol} in our setting.

\begin{figure}[h!]
  \includegraphics[width=0.6\linewidth]{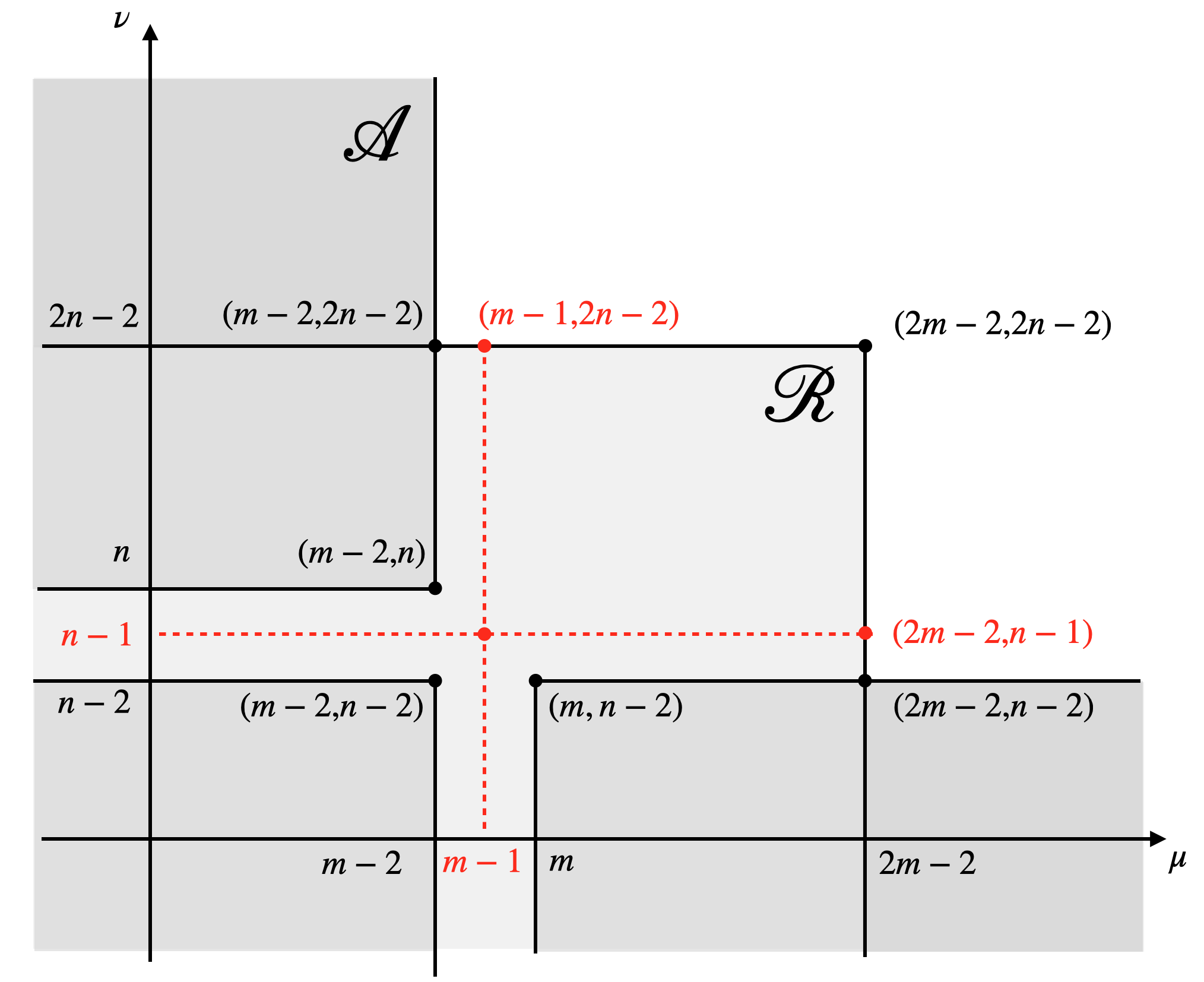}
  \caption{The subsets $\Rc$ (light and dark grey areas) and $\Ar$ (dark grey areas) of $\ZZ^2$, as defined by \eqref{eq:areaR} and \eqref{eq:areaA} respectively, are illustrated. The dashed red lines correspond to bidegrees $(\mu,\nu)$ where some new determinantal representations of the surface $\Sr$ are obtained by introducing quadratic relations.}  
  \label{fig:domains}
\end{figure}

\subsection{Torsion of the symmetric algebra}\label{subsec:torSI}

In order to get new determinantal representations of $\Sr$ we aim to build more compact free resolutions than those obtained in Theorem \ref{thm:botbol}. For that purpose, following ideas introduced in \cite{BCS10} we will consider the symmetric algebra of $I$ modulo its $\mm$-torsion. More precisely, we will consider the algebra $\Sc_I^*:=\Sc_I/H^0_\mm(\Sc_I)$ which is an intermediate image of $\Sc_I$ in the way to get the Rees algebra $\Rc_I$. We have $\Sc_I^*\simeq T/\Jc$ where $\Jc:=\Jc\la 1\ra :\mm^\infty$. Thus, we will insert moving quadrics in some complexes $(\Zc_\bullet)_{(\mu,\nu)}$ by considering the ideal $\Jc\la 2\ra$ which is generated by elements in $\Jc$ of degree at most 2 in the $x_i$'s.

\subsection*{Assumption.} As we will use the algebras $\Sc_I$ and $\Sc_I^*$ in place of the Rees algebra, we need to assume that $\pi(\Gamma')$ is a proper hypersurface in $\PP^3$. So, \emph{in what follows we will always assume that the base points of $\phi$ are all locally generated by at most three generators.}

\section{Local cohomology modules of the symmetric algebra}\label{sec:locSI}

The moving planes following $\phi$ generate the defining ideal $\Jc\la 1\ra\subset T$ of the symmetric algebra $\Sc_I$ and can thus be studied through the properties of this algebra \cite{BuJo03,Bot11}. For their part, the moving quadrics following $\phi$, i.e.~the polynomials in $\Kc$ of degree 2 in the $x_i$'s, are more difficult to analyze. Our main strategy is to study the quotient $\Jc\la 2\ra/\Jc\la 1\ra$ that corresponds to moving quadrics that are not generated by moving planes (we notice that if some base points are a.l.c.i.~then these are not necessarily all the moving quadrics). As this quotient is naturally connected to the study of the local cohomology module $H^0_\mm(\Sc_I)=\Jc/ \Jc\la 1\ra$, this section is devoted to an analysis of the local cohomology modules of $\Sc_I$ in our setting, as well as the corresponding approximation complex of cycles $\Zc_\bullet$.

\medskip

The complex $\Zc_\bullet$ is by definition deeply connected to the Koszul complex $K_\bullet(\fb;T)$, so we begin with two preliminary lemmas on Koszul cycles and homology modules. We denote by $Z_p, B_p$ and $H_p$ the Koszul cycles, boundaries and homology modules respectively. We also remind that the local cohomology modules of the bigraded ring $R=R_1\otimes R_2$ with respect to $\mm$ are such that (see for instance \cite[\S 6]{Bot11}): 
\begin{eqnarray}\label{eq:loc-coho}
& H^i_\mm(R)=0 \textrm{ for any } i\notin  \{2,3\}, \\ \nonumber
& H^2_\mm(R)\simeq \left(\frac{1}{s_0 s_1}k[s_0^{-1},s_1^{-1}]\otimes R_2 \right) \oplus \left( R_1\otimes\frac{1}{t_0 t_1}k[t_0^{-1},t_1^{-1}]\right),\\  \nonumber
& H^3_\mm(R)\simeq \frac{1}{s_0s_1t_0t_1}k[s_0^{-1},s_1^{-1},t_0^{-1},t_1^{-1}].	
\end{eqnarray}
Therefore, since  $K_p$ is a free bigraded $R$-module for any $p$,
 $$K_p\simeq \wedge^p \left(\oplus_{i=1}^4 R(-(m,n))\right)\simeq \oplus_{i=1}^{\binom{4}{p}} R(-pm,-pn)),$$
 the local cohomology modules $H^i_\mm(K_p)$ are easily deduced from $H^i_\mm(R)$ for any $i$.

\begin{lem}\label{lem:HmHp} We have $H_4=H_3=0$, $H^i_\mm(H_p)=0$ for any $i>1$. Moreover, for $j=0,1$ and for any $p$ we have
	$$H^i_\mm(H_p)_{(\mu,\nu)}=0 \ \textrm{ for any } (\mu,\nu) \textrm{ such that }
	H^2_\mm(K_{p+2-i})_{(\mu,\nu)}=	H^3_\mm(K_{p+3-i})_{(\mu,\nu)}=0.$$
\end{lem}
\begin{proof} Since $I$ has depth 2 we deduce that $H_4=H_3$. Moreover, the homology modules $H_p$ are supported in $\Bc$ for any $p$, so their local cohomology modules $H^i_\mm(H_p)$ vanish for any  $i>1$.
	
	Now, consider the double complex $\Cr^\bullet_\mm(K_\bullet)$ obtained by replacing each module of the Koszul complex $K_\bullet$ by its associated Cech complex. The row-filtered spectral sequence associated to $\Cr^\bullet_\mm(K_\bullet)$ converge at the second step and is of the form
	\[
	\xymatrixrowsep{0.2em}
	\xymatrixcolsep{0.2em}
	\xymatrix{
	{0} & {0} & H^0_\mm({H_{2}})& H^0_\mm(H_1) & H^0_\mm ({H_{0}}) \\
	{0} & {0} & H^1_\mm({H_{2}})& H^1_\mm(H_1) & H^1_\mm({H_{0}})  \\
	{0} & {0} & 0& 0 &0\\
	{0} & {0} & 0& 0 &0\\
	}
	\]
On the other hand, the modules of the column-filtered spectral sequence at the first page are the local cohomology modules $H^p_\mm(K_q)$. Since these two spectral sequences have the same abutment, we deduce the claimed property.   	
\end{proof}

\begin{lem}\label{lem:HmZp} The following properties hold:
	\begin{itemize}
		\item[i)] $H^0_\mm(Z_p)=H^1_\mm(Z_p)$ for any $p$,
		\item[ii)] $H^i_\mm(Z_3)\simeq H^i_\mm(K_4)$ for any $i$,
	\end{itemize}	
	and for $p=1,2$,
	\begin{itemize}
		\item[iii)] $H^2_\mm(Z_p)_{(\mu,\nu)}\simeq H^0_\mm(H_{p-1})_{(\mu,\nu)}$ for any $(\mu,\nu)$ such that $H^2_\mm(K_p)_{(\mu,\nu)}=0$,
		\item[iv)] $H^3_\mm(Z_p)_{(\mu,\nu)}=0$ for any $(\mu,\nu)$ such that $$H^2_\mm(K_{p-1})_{(\mu,\nu)}=H^2_\mm(K_p)_{(\mu,\nu)}=H^3_\mm(K_{p+1})_{(\mu,\nu)}=0.$$		
	\end{itemize}
\end{lem}
\begin{proof} The ideal $I$ has depth 2 so $H_4=Z_4=0$ and $H_3=0$, which implies that $Z_3\simeq K_4$. Therefore, we only have to consider $Z_p$ for $p=1,2$.
	Consider the following truncation of the Koszul complex
	$$ \Kc_\bullet \, : \,  0 \rightarrow Z_p \hookrightarrow K_p \rightarrow \cdots \rightarrow K_0$$
and the double complex $\Cr^\bullet_\mm(\Kc_\bullet)$ obtained by replacing each module of $\Kc_\bullet$ by its corresponding Cech complex.
Its row-filtered spectral sequence converges at the second step and is of the form
	\[
	\xymatrixrowsep{0.2em}
	\xymatrixcolsep{0.2em}
	\xymatrix{
	{0} & {0} & H^0_\mm({H_{p-1}})& \cdots & H^0_\mm ({H_{0}}) \\
	{0} & {0} & H^1_\mm({H_{p-1}})& \cdots & H^1_\mm({H_{0}})  \\
	{0} & {0} & 0& \cdots &0\\
	{0} & {0} & 0& \cdots &0\\
	}
	\]
because the homology modules $H_i$ are supported in $\Bc$, which is zero-dimensional, for any $i$.
On the other hand, the first page of its column-filtered spectral sequence is of the form
	$$
	\xymatrixrowsep{0.2em}
	\xymatrixcolsep{1em}
	\xymatrix{
	H^0_\mm(Z_p) &  0  & 0 &   0 \\
	H^1_\mm(Z_p) &  0  &  0 &   0 \\
	H^2_\mm(Z_p) \ar[r] & H^2_\mm(K_p)  \ar[r]&\cdots  \ar[r]&  H^2_\mm(K_0) \\
	H^3_\mm(Z_p) \ar[r] & H^3_\mm(K_p)  \ar[r]&\cdots \ar[r] &  H^3_\mm(K_0) \\}
	$$
Since these two spectral sequences have the same abutment, we deduce that $H^0_\mm(Z_p)=H^1_\mm(Z_p)=0$, that
$H^2_\mm(Z_p)_{(\mu,\nu)}\simeq H^0_\mm(H_{p-1})_{(\mu,\nu)}$ for any $(\mu,\nu)$ such that $H^2_\mm(K_p)_{(\mu,\nu)}=0$ and also that $H^3_\mm(Z_p)_{(\mu,\nu)}=0$ for any ${(\mu,\nu)}$ such that
$$H^3_\mm(K_p)_{(\mu,\nu)}=H^2_\mm(K_{p-1})_{(\mu,\nu)}=0 \textrm{  and  } H^1_\mm(H_{p-1})_{(\mu,\nu)}=H^0_\mm(H_{p-2})_{(\mu,\nu)}=0.$$
Now, applying Lemma \ref{lem:HmHp} we get that $H^1_\mm(H_{p-1})_{(\mu,\nu)}=H^0_\mm(H_{p-2})_{(\mu,\nu)}=0$ for any
${(\mu,\nu)}$ such that $H^2_\mm(K_{p})_{(\mu,\nu)}=H^3_\mm(K_{p+1})_{(\mu,\nu)}=0$ and the conclusion follows by observing that
if $H^3_\mm(K_{p+1})_{(\mu,\nu)}=0$ then necessarily  $H^3_\mm(K_{p})_{(\mu,\nu)}=0$.
\end{proof}

Now, we focus on the local cohomology modules of the symmetric algebra $\Sc_I$. To state our results, we introduce the following subset of $\ZZ^2$:
\begin{multline}\label{eq:areaA}
\mathscr{A}:=\{ (\mu,\nu)\in \ZZ^2 : \mu\leq m-2, \nu\leq n-2 \}\cup 	 \\
\{ (\mu,\nu)\in \ZZ^2 : \mu\geq m,  \nu\leq n-2 \}\cup 	
\{ (\mu,\nu)\in \ZZ^2 : \mu\leq m-2,  \nu\geq n \}
\end{multline}
(see Figure \ref{fig:domains} where $\Ar$ corresponds to the union of the dark grey areas). We also denote by $\delta_{(\mu,\nu)}$ the canonical map
$$\delta_{(\mu,\nu)} : H^2_\mm(\Zc_2)_{(\mu,\nu)} \rightarrow H^2_\mm(\Zc_1)_{(\mu,\nu)}$$ that is induced by the approximation complex $\Zc_\bullet$.

\begin{prop}\label{prop:HmSI} For any couples of non negative integers $(\mu,\nu) \notin \Ar$ the complex ${\Zc_\bullet}_{(\mu,\nu)}$ is acyclic and we have the isomorphisms
	$$H^0_\mm(\Sc_I)_{(\mu,\nu)}\simeq \ker \delta_{(\mu,\nu)}, \ H^1_\mm(\Sc_I)_{(\mu,\nu)}\simeq \coker \delta_{(\mu,\nu)}, \
	H^2_\mm(\Sc_I)_{(\mu,\nu)}=H^3_\mm(\Sc_I)_{(\mu,\nu)}=0.$$
\end{prop}
\begin{proof} By our assumption, the complex $\Zc_\bullet$ is acyclic after localization at any point in $\PP^1\times \PP^1$ (see e.g.~\cite[\S 4]{Bot11}). It follows that for any $p>0$ the homology modules $H_p(\Zc_\bullet)$ are $\mm$-torsion, which implies that $H^i_\mm(H_p(\Zc_\bullet))=0$ for $i>0$ and that  $H^0_\mm(H_p(\Zc_\bullet))=H_p(\Zc_\bullet)$.
	
	Now, consider the double complex $\Cr^\bullet_\mm(\Zc_\bullet)$. We deduce that its row-filtered spectral sequence converges at the second step and is of the form
	\[
	\xymatrixrowsep{0.2em}
	\xymatrixcolsep{0.2em}
	\xymatrix{
	H_3(\Zc_\bullet) & H_2(\Zc_\bullet) & H_1(\Zc_\bullet)  & H^0_\mm (\Sc_I) \\
	{0} & {0} & 0 &  H^1_\mm(\Sc_I)  \\
	{0} & {0} & 0&  H^2_\mm(\Sc_I)\\
	{0} & {0} & 0&  H^3_\mm(\Sc_I)\\
	}
	\]
On the other hand, by Lemma \ref{lem:HmZp} the column-filtered spectral sequence of $\Cr^\bullet_\mm(\Zc_\bullet)$ at the first page is of the form: 	
	\[
	\xymatrixrowsep{0.2em}
	\xymatrixcolsep{1em}
	\xymatrix{
	0 & 0 & 0  & 0 \\
	{0} & {0} & 0 &  0  \\
	H^2_\mm(\Zc_3) \ar[r]& H^2_\mm(\Zc_2) \ar[r]^{\delta_{(\mu,\nu)}}& H^2_\mm(\Zc_1) \ar[r]&  H^2_\mm(\Zc_0)\\
	H^3_\mm(\Zc_3) \ar[r]& H^3_\mm(\Zc_2) \ar[r]& H^3_\mm(\Zc_1) \ar[r]&  H^3_\mm(\Zc_0)\\
	}
	\]
The comparison of these two spectral sequences shows that $H_3(\Zc_\bullet)=H_2(\Zc_\bullet)=0$ and that $H_1(\Zc_\bullet)_{(\mu,\nu)}=0$ for any $(\mu,\nu)$ such that $H^2_\mm(\Zc_3)_{(\mu,\nu)}=0$, and this latter condition holds if $(\mu,\nu) \notin \Ar$ by Lemma \ref{lem:HmZp}, ii). Similarly, applying again Lemma \ref{lem:HmZp} we get that for any $(\mu,\nu) \notin \Ar$, $H^2_\mm(\Zc_0)_{(\mu,\nu)}=0$ and $H^3_\mm(\Zc_p)_{(\mu,\nu)}=0$ for any $p$ (we recall the grading \eqref{eq:Zc}  of $\Zc_\bullet$ we use). Therefore, the graded piece of degree ${(\mu,\nu)}\notin \Ar$ of the first page of this spectral sequence is simply the map $\delta_{(\mu,\nu)}$, which concludes the proof.	
\end{proof}

As a consequence of the above results, we recover, in our specific setting, the results proved by Botbol \cite{Bot11}. In the case where $\phi$ has no base point we also recover the classical determinantal representation of $\Sr$ as the determinant of a matrix of moving planes \cite{DixonRes,CGZ00}.

\begin{cor}\label{cor:Bot} Assume that all the base points are l.c.i. Then, for any $(\mu,\nu)\notin \Rr$ the determinant of the graded complex of free $S$-modules $(\Zc_\bullet)_{(\mu,\nu)}$
is equal to $F^{\deg(\phi)}$, where $F=0$ is a defining equation of $\Sr$.

Moreover, if $(\mu,\nu)=(2m-1,n-1)$ or $(\mu,\nu)=(m-1,2n-1)$ then $(\Zc_\bullet)_{(\mu,\nu)}$ has length two and its determinant is hence the ratio of two determinants of matrices of linear forms.
If in addition $\phi$ has no base point then $(\Zc_\bullet)_{(2m-1,n-1)}$ and $(\Zc_\bullet)_{(m-1,2n-1)}$ have length one and yield $(2mn\times 2mn)$-matrices of linear forms whose determinants are both equal to $F^{\deg(\phi)}$, up to multiplication by a nonzero constant.
\end{cor}
\begin{proof} Applying Lemma \ref{lem:HmZp}, for any $(\mu,\nu)\notin \Rr$ we have $H^2_\mm(\Zc_1)_{(\mu,\nu)}=H^2_\mm(\Zc_2)_{(\mu,\nu)}=0$. Therefore, for any $(\mu,\nu)\notin \Rr$ we get $H^0_\mm(\Sc_I)_{(\mu,\nu)}=0$ and this implies the claimed property by acyclicity of $(\Zc_\bullet)_{(\mu,\nu)}$. We also notice that $\Zc_3\simeq R(-m,-n)\otimes S(-3)$ and hence $(\Zc_3)_{(\mu,\nu)}=0$ if $(\mu,\nu)=(2m-1,n-1)$ or $(\mu,\nu)=(m-1,2n-1)$.
	
Now, to prove the last claim we have to show that $(\Zc_2)_{(\mu,\nu)}=0$ if $\phi$ has no base point and $(\mu,\nu)=(2m-1,n-1)$ or $(\mu,\nu)=(m-1,2n-1)$, which we assume. By definition, $\Zc_2\simeq Z_2[2(m,n)]\otimes S(-2)$ and since the Koszul boundaries $B_2$ are generated in degree $(3m,3n)$ we deduce that $(Z_2)_{(\mu,\nu)+2(m,n)}\simeq (H_2)_{(\mu,\nu)+2(m,n)}$. Now, since $\phi$ has no base point $H_2$ is $\mm$-torsion, which implies that $H_2=H^0_\mm(H_2)$. Finally, applying Lemma \ref{lem:HmHp} we obtain that $(Z_2)_{(\mu,\nu)+2(m,n)}=0$ if $H^2_\mm(K_4)_{(\mu,\nu)+2(m,n)}=0$. But since $K_4\simeq R(-4(m,n))$, we see that $H^2_\mm(K_4)_{(4m-1,3n-1)}\simeq H^2_m(R)_{(-1,n-1)}=0$ and  $H^2_\mm(K_4)_{(3m-1,4n-1)}\simeq H^2_m(R)_{(m-1,-1)}=0$, which concludes the proof.
\end{proof}

\begin{rk} If the assumption on the base points is weakened so that some base points are a.l.c.i.~then Corollary \ref{cor:Bot} is still valid with the difference that the determinants contain linear forms as extraneous factors; more precisely, with the notation introduced in Section \ref{subsec:MPSI} they are all equal to
	$$F(\xb)^{\deg{\phi}}\cdot \prod_{p \in \Bc, \textrm{ a.l.c.i.}} L_p(\xb)^{e_p-d_p},$$
up to multiplication by a nonzero constant (see \cite{BCJ09} and \cite[Lemma 5.6]{Bot11}).	
\end{rk}

In order to obtain more compact representations of $\Sr$ than those obtained in Corollary \ref{cor:Bot}, Proposition \ref{prop:HmSI} suggests to investigate the map $\delta_{(\mu,\nu)}$ since it is strongly connected to the module $H^0_\mm(\Sc_I)\simeq \Jc/\Jc\la 1\ra$. Moreover, the shape of the area $\Ar$, where the complex  $\Zc_\bullet$ is acyclic, suggests to focus on bidegrees of the form $(*,n-1)$ and $(m-1,*)$: those bidegrees are marked with dashed red lines in Figure \ref{fig:domains}.

\emph{In what follows, for the sake of simplicity we will only focus on the bidegrees of the form $(*,n-1)$; the symmetric case of bidegrees of the form $(m-1,*)$ can be treated exactly in the same way.
}

\section{Moving quadrics and torsion of the symmetric algebra}\label{sec:MQ-SI}

In this section, we will combine moving quadrics and moving planes in order to obtain graded free $S$-resolutions of $(\Sc_I^*)_{(\mu,\nu)}$ for some ${(\mu,\nu)} \in \Rr$, hence of $(\Rc_I)_{(\mu,\nu)}$ if all base points are l.c.i. A key ingredient in our approach is the study of the module of moving planes of degree $(*,n-1)$.

\medskip

In what follows, the Hilbert function of a bigraded $R$-module $M$ is denoted by
$\HF_{M}:\ZZ^2 \rightarrow \NN.$
It is defined by $\HF_M(\mu,\nu)=\dim_k M_{(\mu,\nu)}$. Its corresponding Hilbert polynomial is denoted by $\HP_M$ and we recall that the following Grothendieck-Serre formula holds (see e.g.~\cite[Proposition 4.26]{BoCh17} or \cite{Hyry99}): for any $(\mu,\nu) \in \ZZ^2$ we have
$$ \HF_M(\mu,\nu) -\HP_M(\mu,\nu)= \sum_{i\geq 0} (-1)^i \HF_{H^i_\mm(M)}(\mu,\nu).$$

Recall that $Z_1$ is the first module of syzygies of $\fb=(f_0,\ldots,f_3)$ whose grading is induced by the canonical grading of the Koszul complex $K_\bullet(\fb;R)$; more precisely $Z_1\subset K_1\simeq R(-m,-n)^4$. Now, consider the syzygies of $\fb$ whose degree with respect to $R_2$ is equal to $n-1$:
$$\overline{Z}_1:=\oplus_{\mu \geq 0} {Z_1}_{(\mu,2n-1)}.$$
It is a graded $R_1=k[s_0,s_1]$-module that fits in the exact sequence of graded $R_1$-modules
$$ 0 \rightarrow \overline{Z}_1 \rightarrow \oplus_{i=1}^4 R_1(-m)\otimes (R_2)_{n-1} \xrightarrow{f_0,\ldots,f_3} R_1\otimes (R_2)_{2n-1},$$
where the notation $(R_2)_{l}$ refers to the graded piece of $R_2$ of degree $l$, which is isomorphic to $k^{l+1}$.
By Hilbert Syzygy Theorem $\overline{Z}_1$ is a free module. Moreover, its rank is equal to $2n$ and hence there exist non negative integers $\mu_1,\ldots,\mu_{2n}$ such that
\begin{equation}\label{eq:Z1bar}
\overline{Z}_1 \simeq \oplus_{i=1}^{2n} R_1(-m-\mu_i).	
\end{equation}

\begin{definition} We define the \emph{threshold degree} $\mu_0:=\max_{1\leq i \leq 2n} {\mu_i}$. It is the maximum degree of a minimal syzygy of $f_0,f_1,f_2,f_3$ of degree $n-1$ with respect to $R_2$.
\end{definition}

The threshold degree $\mu_0$ provides a fine measure of the values of $\mu$ such that the number of linearly independent moving planes of degree $(\mu-1,n-1)$ reaches its expected value, equivalently the Hilbert function of $H_0=R/I$ reaches its Hilbert polynomial. More precisely, from Lemma \ref{lem:HmHp} and the description \eqref{eq:loc-coho} of the local cohomology modules of $R$, we deduce straightforwardly that for any integer $\mu \geq 2m$  $$H^0_\mm(H_0)_{(\mu-1+m,2n-1)}=H^1_\mm(H_0)_{(\mu-1+m,2n-1)}=0$$ and hence, by the Grothendieck-Serre formula,
$$\HF_{H_0}(\mu-1+m,2n-1)=r:=\sum_{p\in \Bc} d_p.$$
The threshold integer $\mu_0$ is nothing that the smallest integer $\mu$ such that the three above equalities holds:

\begin{prop}\label{prop:freeness} For any $\mu\geq \mu_0$ the number of linearly independent moving planes of degree $(\mu-1,n-1)$ is equal to $$\HF_{Z_1}(\mu-1+m,2n-1)=2n(\mu-m)+r$$	
and we have
	$$\HF_{H_0}(\mu-1+m,2n-1)=r, \ H^0_\mm(H_0)_{(\mu-1+m,2n-1)}=H^1_\mm(H_0)_{(\mu-1+m,2n-1)}=0$$
and
$$m-\frac{r}{2n} \leq \mu_0\leq \min\{2m,2mn-r\}.$$	
\end{prop}
\begin{proof} By definition of $\mu_0$ and \eqref{eq:Z1bar}, for any $\mu\geq \mu_0$
	$$\HF_{Z_1}(\mu-1+m,2n-1)=\sum_{i=1}^{2n}(\mu-\mu_i)=2n\mu-\sum_{i=1}^{2n}\mu_i.$$	
	Therefore, from the exact sequence of bigraded $R$-modules
	$$ 0 \rightarrow Z_1 \rightarrow \oplus_{i=1}^4 R(-m,-n) \xrightarrow{f_0,\ldots,f_3} R \rightarrow H_0 \rightarrow 0$$
we deduce that for any $\mu\geq \mu_0$
\begin{align*}
	\HF_{H_0}(\mu-1+m,2n-1) &= \HF_{Z_1}(\mu-1+m,2n-1) -4\HF_{R}(\mu-1,n-1)\\
	                        & \hspace{1em} +\HF_{R}(\mu-1+m,2n-1) \\
							&= 2n\mu-\sum_{i=1}^{2n}\mu_i -4\mu n+2n(\mu+n) = 2mn-\sum_{i=1}^{2n}\mu_i.
\end{align*}
But since, as we already noticed, $\HF_{H_0}(\mu-1+m,2n-1)=r$ for $\mu\geq 2m$ we deduce that
$r=2mn-\sum_{i=1}^{2n}\mu_i$ and the claimed properties about the Hilbert functions of $Z_1$ and $H_0$ follow. Notice that we also deduce that $\mu_0\leq 2m$. In addition, Lemma \ref{lem:HmHp} shows that $H^1_\mm(H_0)_{(\mu,\nu)}=0$ for any $(\mu,\nu)\notin \Rr$, so we have $H^1_\mm(H_0)_{(\mu-1+m,2n-1)}=0$ for any  $\mu\geq 0$. From here the claimed vanishing property of $H^0_\mm(H_0)$ follows from the Grothendieck-Serre formula. Finally, the lower bound on $\mu_0$ follows straightforwardly from the equality $2mn-r=\sum_{i=1}^{2n}\mu_i$.
\end{proof}

We already noticed that $H^0_\mm(\Sc_I)_{(\mu-1,n-1)}=0$ for any $\mu$ such that $(\mu-1,n-1)\notin \Rr$, i.e. for any $\mu \geq 2m$ (see the proof of Corollary \ref{cor:Bot}). Proposition \ref{prop:freeness} allows us to give a structural property of the graded pieces of this local cohomology module for any  $\mu \geq \mu_0$. We recall from Section \ref{subsec:torSI} that the ideal $\Jc\la 1\ra$ is the defining ideal of $\Sc_I$, $\Jc=\Jc\la 1\ra:\mm^\infty$ the defining ideal of $\Sc_I^*$ and $\Jc\la 2 \ra$ is the ideal generated by elements in $\Jc$ of degree at most 2 in the variables $\xb$. 	

\begin{prop}\label{prop:J2/J1} For any integer $\mu\geq \mu_0$, $H^0_\mm(\Sc_I)_{(\mu-1,n-1)}$ is a free $S$-module and we have the isomorphism of graded $S$-modules
$$H^0_\mm(\Sc_I)_{(\mu-1,n-1)}\simeq  (\Jc\la 2\ra/\Jc\la 1\ra )_{(\mu-1,n-1)}.$$
Moreover, $H^1_\mm(\Sc_I)_{(\mu-1,n-1)}=0$ for any $\mu\geq \mu_0$.
\end{prop}
\begin{proof}
By Lemma \ref{lem:HmZp}, iii) we get
$$H^2_\mm(\Zc_1)_{(\mu-1,n-1)}\simeq H^0_\mm(H_0)_{(\mu-1,n-1)+(m,n)}\otimes S(-1)$$
for any integer $\mu$. Therefore, Proposition \ref{prop:freeness} implies that $H^2_\mm(\Zc_1)_{(\mu-1,n-1)}=0$ for any $\mu\geq \mu_0$. Moreover, applying Proposition \ref{prop:HmSI} we deduce that for any $\mu\geq \mu_0$ we have
\begin{equation*}
H^0_\mm(\Sc_I)_{(\mu-1,n-1)}\simeq H^2_\mm(\Zc_2)_{(\mu-1,n-1)}\simeq H^2_\mm(Z_2)_{(\mu-1,n-1)+2(m,n)}\otimes S(-2).	
\end{equation*}
and $$H^1_\mm(\Sc_I)_{(\mu-1,n-1)}=0.$$
It follows that $H^0_\mm(\Sc_I)_{(\mu-1,n-1)}$ is a free $S$-module, generated in degree 2.
\end{proof}

In the terminology of moving planes and moving quadrics following $\phi$, we just proved that for any $\mu\geq \mu_0$,
$H^0_\mm(\Sc_I)_{(\mu-1,n-1)}$ is the quotient of the vector space $W_{(\mu,\nu)}$ of moving quadrics of bidegree $(\mu-1,n-1)$ by the vector subspace $V_{(\mu,\nu)}'$ of moving quadrics generated by moving planes of the same bidegree. In addition, we have to the following result.

\begin{thm}\label{thm:quadres} For any $\mu\geq \mu_0$ the complex
\begin{equation}\label{eq:Z1Z2cpx}
	0\rightarrow (\Zc_2)_{(\mu-1,n-1)} \rightarrow (\Zc_1)_{(\mu-1,n-1)}\oplus \left(\Jc\la 2 \ra/\Jc\la 1 \ra\right)_{(\mu-1,n-1)}\rightarrow (\Zc_0)_{(\mu-1,n-1)}
\end{equation}	
is a minimal graded free resolution of the $S$-module $(\Sc_I^*)_{(\mu-1,n-1)}$. Moreover, its determinant is equal to
$$ F(\xb)^{\deg(\phi)}\cdot\prod_{\stackrel{p\in \Bc}{p\ a.l.c.i}} L_p(\xb)^{e_p-d_p},$$
where $F(\xb)=0$ is a defining equation of $\Sr$.
\end{thm}
\begin{proof} This is a consequence of Proposition \ref{prop:HmSI}, which shows the acyclicity of the complex $(\Zc_\bullet)_{(\mu-1,n-1)}$, and Proposition \ref{prop:J2/J1} which shows that the $S$-module $\left(\Jc\la 2 \ra/\Jc\la 1 \ra\right)_{(\mu-1,n-1)}$ is free. The irreducible decomposition of the determinants of these complexes is discussed in Section \ref{subsec:MPSI}.
\end{proof}

As an immediate corollary, we get the following free resolutions of some graded components of the Rees algebra of $I$.

\begin{cor}\label{cor:corthquad} Assume that all the base points are l.c.i. Then for any $\mu\geq \mu_0$ the complex \eqref{eq:Z1Z2cpx} is a minimal graded free resolution of $(\Rc_I)_{(\mu-1,n-1)}$. Its determinant yields a representation of $\Sr$, raised to the power $\deg(\phi)$, as the ratio of two determinants, the one in the numerator being filled with moving planes and quadrics following $\phi$.
\end{cor}
\begin{proof} This follows from the fact that $\Sc_I^*=\Rc_I$ if and only if all the base points are l.c.i.
\end{proof}

\noindent {\it Proof of Theorem \ref{thm:intro2}:} Theorem \ref{thm:quadres} and Corollary \ref{cor:corthquad} prove Theorem \ref{thm:intro2} since by definition the matrices of \eqref{eq:Z1Z2cpx} are from left to right, a matrix of linear forms built from the second order Koszul syzygies of $I$ in bidegree $(\mu-1+2m,3n-1)$, including a block of zero rows at its bottom, and the matrix $\MM_{(\mu-1,n-1)}$ of moving planes and quadrics defined in Section \ref{sec:intro}. \qed
\medskip

We notice that if $\mu=2m$ the complex \eqref{eq:Z1Z2cpx} gives back the complex described in Corollary \ref{cor:Bot} since in this case $H^0_\mm(\Sc_I)_{(2m-1,n-1)}=0$. This means that the column block of $\MM_{(\mu-1,n-1)}$ corresponding to moving quadrics disappears. Actually, the family of complexes \eqref{eq:Z1Z2cpx} can be interpreted as an extension of the family of complexes given by Botbol for any $(\mu,\nu)\notin \Rr$ (the white zone in Figure \ref{fig:domains}; see Theorem \ref{thm:botbol}) along the horizontal red dashed lines plotted in Figure \ref{fig:domains} (entering the grey area from the white area).

One may wonder for which value of $\mu$ moving quadrics do appear in \eqref{eq:Z1Z2cpx}. Let $l_\mu$ and $q_\mu$ be the ranks of the free $S$-modules $(\Zc_1)_{(\mu-1,n-1)}$ and $\left(\Jc\la 2 \ra/\Jc\la 1 \ra\right)_{(\mu-1,n-1)}$ respectively. By Proposition \ref{prop:freeness} we have $l_\mu=2n(\mu-m)+r$ for any $\mu\geq \mu_0$. Since we must have $l_\mu+q_\mu\geq \dim_k R_{(\mu-1,n-1)}=\mu n$ for any $\mu\geq \mu_0$, we deduce that $q_\mu\geq n(2m-\mu)-r$. In particular, we necessarily have $q_\mu>0$, i.e.~moving quadrics do appear in \eqref{eq:Z1Z2cpx}, for any $\mu$ such that
\begin{equation}\label{eq:qmu}
m-\frac{r}{2n}\leq \mu_0\leq \mu < 2m-\frac{r}{n}.	
\end{equation}

\section{Determinantal formulas}\label{sec:detformula}

The goal of this section is to prove Theorem \ref{thm:intro1}. For that purpose, we analyze under which condition the resolution obtained in Theorem \ref{thm:quadres} has length one, equivalently ${\Zc_2}_{(\mu-1,n-1)}=0$, and hence yields a determinantal representation of the surface $\Sr$.

\medskip

Let $\fb'=\{f_0',f_1',f_2'\}$ be three general $k$-linear combinations of the polynomials $f_0,\ldots,f_3$ and set $I'=(\fb')$. Since all the base points are assumed to be locally defined by at most three generators, $I$ and $I'$ have the same saturation with respect to $\mm$, i.e.~they both define the base scheme $\Bc$. We denote by $K'_\bullet$ the Koszul complex of $\fb':=\{f_1',f_2',f_3'\}$ over $R$ and by $Z_p',B_p'$ and $H_p'$ its cycles, boundaries and homology modules respectively.

It turns out that the existence of syzygies of $\fb'$ of degree $(\mu-1,n-1)$ is strongly related to the vanishing of $(\Zc_2)_{(\mu-1,n-1)}$. Conditions on the existence of syzygies of $\fb'$ already appeared in the literature on the method of moving quadrics (see e.g.~\cite{CGZ00,BCD03,AHW05}) and led to the development of structural results on Koszul syzygies of $\fb'$, i.e.~on $B_1'$; see \cite{Cox01,CoSc03}, \cite[Section 4.3]{BuJo03} and \cite{HoWa06,AHW05}. The next result is another illustration of the importance of this concept.

\begin{prop}\label{prop:Z2Z1p} If there is no syzygies of $\fb'$ in degree $(\mu-1,n-1)$ for some integer $\mu$, then ${\Zc_2}_{(\mu-1,n-1)}=0$. Moreover, if all the base points are l.c.i.~then this property is an equivalence.
\end{prop}
\begin{proof}
Since the $f_i$'s are minimal generators of $I$, they form a $k$-basis of $I_{(m,n)}$ and hence $I'=I+(h)$ for some $h\in I_{(m,n)}$. Thus, the Koszul homology modules $H_p$ of the Koszul complex $K_\bullet(\fb;R)$ are isomorphic to the Koszul homology modules $H_p(\fb',h;R)$ of the Koszul complex $K_\bullet(\fb',h;R)$. Moreover, by classical properties of Koszul complexes (see \cite[Section 1.6]{BrHe93}) one has the exact sequence of complexes
$$ 0\rightarrow K_\bullet(\fb';R) \xrightarrow{\iota_\bullet} K_\bullet(\fb',h;R) \xrightarrow{\pi_\bullet} K_\bullet(\fb';R)[-m,-n] \rightarrow 0$$
where $\iota$ is the canonical inclusion and $\pi$ the canonical projection. It yields the long exact sequence (\cite[Corollary 1.6.13]{BrHe93}):
\begin{equation}\label{eq:koszulseq}
\cdots \rightarrow H_2' \xrightarrow{\overline{\iota_2}} H_2 \xrightarrow{\overline{\pi_2}} H_1'[-m,-n]
\xrightarrow{\cdot(-h)} H_1' \rightarrow \cdots.	
\end{equation}
Now, since $I'$ has depth 2 we have $H_3'=H_2'=0$. Moreover, since $B_1'$ is generated in degree $(2m,2n)$, for any $\mu$ we have
$$(H_1')[-m,-n]_{(\mu-1+2m,3n-1)}={H_1'}_{(\mu-1+m,2n-1)}={Z_1'}_{(\mu-1+m,2n-1)}.$$
Therefore, if ${Z_1'}_{(\mu-1+m,2n-1)}=0$ then we deduce that ${H_2}_{(\mu-1+2m,3n-1)}=0$. But again, since $B_2$ is generated in degree $(3m,3n)$ we have
$${H_2}_{(\mu-1+2m,3n-1)}={Z_2}_{(\mu-1+2m,3n-1)}$$ for any $\mu$. Therefore, we proved if there is no syzygies of $\fb'$ in degree $(\mu-1,n-1)$ then
${\Zc_2}_{(\mu-1,n-1)}$, as claimed.

Now, we turn to the second claim in this proposition, so we assume that all the base points are l.c.i. In this case, one can assume that $h$ belongs to the saturation ${I'}^{\textrm{sat}}$ of $I'$ with respect to $\mm$; see \cite[Lemma 4.2]{AHW05}. Let $\mu$ be an integer and consider the following graded piece of the multiplication map in \eqref{eq:koszulseq}:
\begin{equation}\label{eq:nulmap}
(H_1'[-m,-n])_{(\mu-1+2m,3n-1)}=(H_1')_{(\mu-1+m,2n-1)}
\xrightarrow{\cdot(-h)} (H_1')_{(\mu-1+2m,3n-1)}.	
\end{equation}
This map send a syzygy $(g_0',g_1',g_2')$ of $I'$ to the syzygy $(-hg_0',-hg_1',-hg_2')$. Since $h$ belongs to ${I'}^{\textrm{sat}}$, this latter syzygy is said to vanish at all the base points, which means that all its components $hg_i'$ belong to ${I'}^{\textrm{sat}}$ (see \cite[Definition 1.5]{HoWa06}). Then, by \cite[Corollary 3.15]{HoWa06} (see also \cite[Theorem 3.10]{AHW05}) we deduce that the syzygy $(-hg_0',-hg_1',-hg_2')$ is a Koszul syzygy, i.e.~that it belongs to $B_1'$. To apply this result it is important to notice that the syzygy $(-hg_0',-hg_1',-hg_2')$ belongs to $(Z_1')_{(\mu-1+2m,3n-1)}$, so its components are polynomials of degree $(\mu-1+m,2n-1)$ in $R$, which allows us to apply \cite[Corollary 3.15]{HoWa06} as the degree with respect to $R_2$ is precisely equal to $2n-1$. These considerations show that the map \eqref{eq:nulmap} is the null map for any integer $\mu$. Moreover, both modules $B_2$ and $B_1'$ are null in the degree we are considering and hence we deduce that for any integer $\mu$ we have (recall that $H_2'=0$):
$$ {Z_2}_{(\mu-1+2m,3n-1)}\simeq {Z_1'}_{(\mu-1+m,2n-1)}.$$
From here, the claimed equivalence follows by taking into account the shifts in the gradings.
\end{proof}

The structure of the syzygies of $I'$ that are of degree $n-1$ with respect to $R_2$ can be analyzed as we did in Section \ref{sec:MQ-SI} for the syzygies of $I$. Consider the module
$${\overline{Z}_1}':=\oplus_{\mu \geq 0} {Z_1'}_{(\mu,2n-1)}.$$
It is a graded $R_1=k[s_0,s_1]$-module that fits in the exact sequence
$$ 0 \rightarrow {\overline{Z}_1}' \rightarrow \oplus_{i=1}^3 R_1(-m)\otimes (R_2)_{n-1} \xrightarrow{f_0',f_1',f_2'} R_1\otimes (R_2)_{2n-1}.$$
The module ${\overline{Z}_1}'$ is free of rank $n$ and hence there exist non negative integers $\mu_1',\ldots,\mu_{n}'$ such that
\begin{equation}\label{eq:Z1pbar}
{\overline{Z}_1}' \simeq \oplus_{i=1}^{n} R_1(-m-\mu_i').	
\end{equation}

\begin{definition} We set $\eta_0:=\min_{1\leq i \leq n} \mu_i'$; this is the largest integer such that there is no nonzero syzygy of $\fb'$ of degree $(\mu-1,n-1)$ for any $\mu\leq \eta_0$.
\end{definition}

\begin{lem}\label{lem:eta0} We have $2m-r \leq \eta_0 \leq 2m - \frac{r}{n}$.
\end{lem}
\begin{proof} We first begin by examining the Hilbert function of $H_0'$. By similar results as the ones in Lemma \ref{lem:HmZp}, we get that
	$H^0_\mm(H_0')_{(\mu,\nu)}=0$ if $H^2_\mm(K_2')_{(\mu,\nu)}=H^3_\mm(K_3')_{(\mu,\nu)}=0$. In the same way, $H^1_\mm(H_0')_{(\mu,\nu)}=0$ if $H^3_\mm(K_2')_{(\mu,\nu)}=0$. Gathering these two conditions, we deduce from \eqref{eq:loc-coho} that for any integer $\mu\geq 2m$ we have
	$$ H^0_\mm(H_0')_{(\mu-1+m,2n-1)}=H^1_\mm(H_0')_{(\mu-1+m,2n-1)}=0$$
	and hence $\HF_{H_0'}(\mu-1+m,2n-1)=r$ by the Grothendieck-Serre formula.
	
	Now, from the exact sequence of bigraded $R$-modules
		$$ 0 \rightarrow Z_1' \rightarrow \oplus_{i=1}^3 R(-m,-n) \xrightarrow{f_0',f_1',f_2'} R \rightarrow H_0' \rightarrow 0$$
we deduce that
	\begin{multline}\label{eq:HFHOp}
		\HF_{Z_1'}(\mu-1+m,2n-1) -3\HF_{R}(\mu-1,n-1) \\ + \HF_{R}(\mu-1+m,2n-1) - \HF_{H_0'}(\mu-1+m,2n-1)=0.
	\end{multline}
	Therefore, using \eqref{eq:Z1pbar} and taking $\mu\gg 0$ we deduce that
	$$\sum_{i=1}^n \mu_i'=2mn -r.$$
	This equality implies that $\eta_0 \leq 2m-\frac{r}{n}$.
	
	To conclude, the equality \eqref{eq:HFHOp} shows that $\HF_{Z_1'}(\mu-1+m,2n-1)$ is a polynomial for any $\mu\geq 2m$ because for those values of $\mu$ the term
	$\HF_{H_0'}(\mu-1+m,2n-1)$ is a (constant) polynomial function. From \eqref{eq:Z1pbar} it follows that $\max_i \mu_i'\leq 2m$. But since $\sum_{i=1}^n \mu_i'=2mn-r$ we deduce that $\eta_0\geq 2m-r$.

\end{proof}

We are now ready to state the main result of this section.

\begin{thm}\label{thm:main} If $\mu_0\leq\eta_0$, then for any integer $\mu$ such that
	$$m-\frac{r}{2n}\leq \mu_0 \leq \mu\leq \eta_0 \leq {2m-\frac{r}{n}}$$
	the complex \eqref{eq:Z1Z2cpx} is reduced to the map of free $S$-modules
	$$0\rightarrow (\Zc_1)_{(\mu-1,n-1)}\oplus \left(\Jc\la 2 \ra/\Jc\la 1 \ra\right)_{(\mu-1,n-1)}\rightarrow (\Zc_0)_{(\mu-1,n-1)}$$
	whose determinant is equal to
$$ F(\xb)^{\deg(\phi)}\cdot\prod_{\stackrel{p\in \Bc}{p\ a.l.c.i}} L_p(\xb)^{e_p-d_p},$$
where $F(\xb)=0$ is a defining equation of $\Sr$. In particular, if all base points are l.c.i.~we obtain a determinantal representation of $\Sr$ raised to the power $\deg(\phi)$.
\end{thm}
\begin{proof} Combine Theorem \ref{thm:quadres}, Proposition \ref{prop:Z2Z1p} and Lemma \ref{lem:eta0}.
\end{proof}

The matrix of the map considered in Theorem \ref{thm:main} is the matrix $\MM_{(\mu-1,n-1)}$ which is composed of $l_\mu=2n(\mu-n)+r$ moving planes $q_\mu=n(2m-\mu)-r$ moving quadrics following $\phi$. We notice that in the presence of a.l.c.i.~base points, one has to be careful about the definition of the space of moving quadrics. Indeed, the formalism we developed assumes that moving quadrics are elements in $\Jc$, i.e.~defining equations of $\Sc_I^*$, and not necessarily in $\Kc$, i.e.~defining equations of the Rees algebra $\Rc_I$. It turns out that these two ideals coincide if all the base points are l.c.i., but otherwise they can be different.

Below, we provide three examples to illustrate both Theorem \ref{thm:main} and Theorem \ref{thm:quadres}. In particular, we will show that the inequality $\mu_0\leq \eta_0$ does not always hold, although it seems to be quite often satisfied from our experiments, a very interesting property that deserves more investigations that we plan to do in a future work. But before that, we explain how Theorem \ref{thm:main} covers the previously known results on the validity of the method of moving quadrics. 

\medskip

\paragraph{\it No base points}
If there is no base points, i.e.~if $r=0$, then $\eta_0=2m$. It follows that we always have $\mu_0\leq \eta_0$, so that there always exists a determinantal representation of $\Sr$ (raised to the power $\deg(\phi)$). We recover here the determinantal representation obtained in \cite{LAI20191}. We emphasize that the method of moving quadrics was first introduced in \cite{CGZ00} where the authors proved that there exists a determinantal representation built solely from moving quadrics in bidegree $(m-1,n-1)$ providing there is no moving planes in this bidegree. In our context, this latter condition implies that $\mu_0=m$ and hence we recover this result but also get in addition that there is a family of determinantal representations that correspond to the values $\mu$ such that $m\leq \mu \leq 2m$. The case $\mu=m$ gives a matrix filled exclusively with moving quadrics, the one given in \cite{CGZ00}, and the case $\mu=2m$ gives a matrix filled exclusively with moving planes, the one given by Dixon \cite{DixonRes}; in between these two extreme values on gets mixed matrices of moving planes and quadrics.

\medskip

\paragraph{\it Ruled surfaces} 
The case $n=1$ corresponds to tensor product surfaces $\Sr$ that are ruled surfaces. In \cite{CZS01} it is proved the following result: let $p$ and $q$ be a basis of the syzygy module $\overline{Z}_1$ and denote also by $L_p$ and $L_q$ their corresponding moving planes; they are of degree $2m-r-\mu_0$ and $\mu_0$ respectively. Then, the Sylvester resultant matrix of $L_p$ and $L_q$ yields a determinantal representation of $\Sr$ (raised to the power $\deg(\phi)$). In our context, since $n=1$ we get $\eta_0=2m-r$ and hence we always have $\mu_0\leq \eta_0$. Thus, the matrix given in \cite{CZS01} corresponds to the case $\mu=\eta_0=2m-r$. We notice that if $\mu_0<\eta_0$ then we also get smaller matrices with moving quadrics; these matrices actually correspond to the Hybrid B\'ezout resultant matrices associated to $L_p$ and $L_q$ for the values $\mu_0\leq \mu\leq 2m-r$.
	
\medskip

\paragraph{\it Results on the matrix $\MM_{(m-1,n-1)}$}  
In \cite{AHW05} the authors prove the validity of the method of moving quadrics under several assumptions; see \cite[Section 4]{AHW05} for more details. Translated in our framework, they assume that
	$m\leq \eta_0$, $\mu_0\leq m$ and they consider the matrix obtained with $\mu=m$, i.e.~the matrix $\MM_{(m-1,n-1)}$. This case is hence a direct consequence of Theorem \ref{thm:main}. We also notice that the assumption $\eta_0\geq m$ implies, by Lemma \ref{lem:eta0}, that $r\leq mn$.

\begin{ex} Consider the map $\phi$ defined by the following polynomials:
\begin{align*}
      f_0&=-s_0^{3}t_0^{2}-2\,s_1^{3}t_0^{2}+s_0^{3}t_1^{2}+2\,s_1^{3}t_1^{2},\\
	  f_1&=4\,s_0^{3}t_0\,t_1+8\,s_1^{3}t_0\,t_1,\\
	  f_2&=s_0\,s_1^{2}t_0^{2}+s_0\,s_1^{2}t_1^{2},\\
	  f_3&=2\,s_1^{3}t_0^{2}+2\,s_1^{3}t_1^{2}.
\end{align*}
We have $m=3$ and $n=2$. Using {\tt Macaulay2} \cite{M2}, we see that this map parameterizes a tensor-product surface of degree 6, that it is generically injective and that its base scheme has degree 6. We also get $\mu_0=2$ and $\eta_0=3$, and $\nu_0=\zeta_0=2$ in the symmetric case. Therefore, Theorem \ref{thm:main} yields two determinantal formula for $\Sr$. The first one is obtained in bidegree $(1,1)$ and is filled with two moving planes and two moving quadrics:
$$\begin{pmatrix}
      x_3&0&0&4 x_2^{2}\\
      {-2 x_2}&0&x_1 x_3&2 x_0 x_3+2 x_3^{2}\\
      0&x_3&{-4 x_2^{2}}&0\\
      0&{-2 x_2}&2 x_0 x_3-2 x_3^{2}&{-x_1 x_3}\end{pmatrix}$$
The second one is obtained in degree $(2,1)$ and is made exclusively of moving planes:
$$
\begin{pmatrix}
      x_3&0&0&0&0&2 x_2\\
      {-2 x_2}&0&x_3&0&0&0\\
      0&0&{-2 x_2}&0&x_1&2 x_0+2 x_3\\
      0&x_3&0&0&{-2 x_2}&0\\
      0&{-2 x_2}&0&x_3&0&0\\
      0&0&0&{-2 x_2}&2 x_0-2 x_3&{-x_1}\end{pmatrix}$$
This latter matrix, fully filled with linear forms, falls in the cases that are covered by the results in \cite{AHW05}. We also notice that Theorem \ref{thm:botbol} yields its more compact representation of $\Sr$ in degree $(5,1)$. Therefore, we see that Theorem \ref{thm:quadres} extends its validity to the additional degrees $(1,1), (2,1),(3,1)$ and $(4,1)$, moving quadrics being needed only in bidegree $(1,1)$ and determinantal formulas being obtained in bidegree $(1,1)$ and $(2,1)$.
\end{ex}

\begin{ex} Consider the map $\phi$ defined by the polynomials
\begin{align*}
	      f_0 &=s_1^{2}t_0^{2}+s_1^{2}t_0\,t_1+s_0\,s_1\,t_1^{2},\\
		  f_1 &=s_1^{2}t_0^{2}+s_0^{2}t_0\,t_1+s_1^{2}t_0\,t_1+2\,s_0\,s_1\,t_1^{2},\\
		  f_2 &=s_0\,s_1\,t_0^{2}+s_0^{2}t_0\,t_1+s_0\,s_1\,t_1^{2}+s_1^{2}t_1^{2},\\
		  f_3&= s_0^{2}t_0^{2}+s_1^{2}t_0^{2}+s_1^{2}t_0\,t_1+s_0\,s_1\,t_1^{2}.
\end{align*}
We have $m=n=2$. Using {\tt Macaulay2} \cite{M2}, we see that this map parameterizes a tensor-product surface of degree 7, that it is generically injective and has a single base point of coordinates $(1:0)\times(0:1)$. We also get $\mu_0=3$ and $\eta_0=3$. Therefore, Theorem \ref{thm:main} shows that the matrix $\MM_{(2,1)}$ is a determinantal representation of $\Sr$; it is equal to
\begin{center}
  \includegraphics[width=1\linewidth]{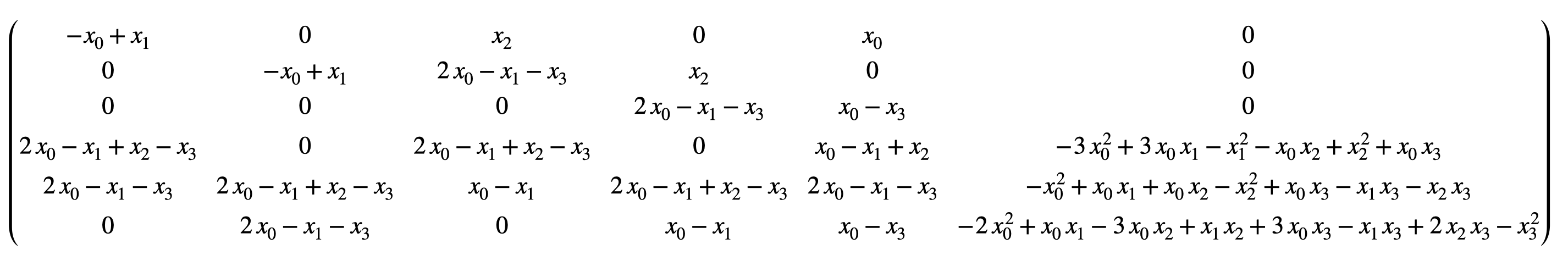}	
\end{center}

\noindent We notice that this example does not fit in the previously known cases we have listed right after the proof of Theorem \ref{thm:main}. 
\end{ex}

\begin{ex}\label{ex:length2} Consider the rational map $\phi$ defined by the polynomials $f_0,f_1,f_2,f_3$ obtained as the 3-minors of the following matrix:
		$$\left(\begin{array}{ccc}
		       s_0&0&s_1^{2}t_0^{2}+s_0\,s_1\,t_1^{2}\\
		       s_1&t_0-t_1&s_0\,s_1\,t_0^{2}+s_0^{2}t_1^{2}\\
		       s_0+s_1&t_0&s_1^{2}t_0\,t_1+s_0\,s_1\,t_1^{2}\\
		       0&t_1&s_0^{2}t_0\,t_1+s_1^{2}t_1^{2}
			   \end{array}\right).$$	
We have $m=3,n=3$ and computations with {\tt Macaulay2} \cite{M2} show that $r=13$, $\deg \Sr=5$ and that all the base points are locally complete intersections. We also get $\mu_0=2$ and $\eta_0=1$ (and $\nu_0=2$, $\zeta_0=1$ in the symmetric case). Therefore, Theorem \ref{thm:main} does not apply. Nevertheless, the complex \eqref{eq:Z1Z2cpx} in degree $(\mu_0-1,n-1)=(1,2)$ is of the form 
$$ 0 \rightarrow S^2(-2) \xrightarrow{\partial_2} S^7(-1)\oplus S^1(-2)\xrightarrow{\partial_1} S^6$$
where the matrices of $\partial_2$ and $\partial_1$ are respectively:
\begin{center}
  \includegraphics[width=1\linewidth]{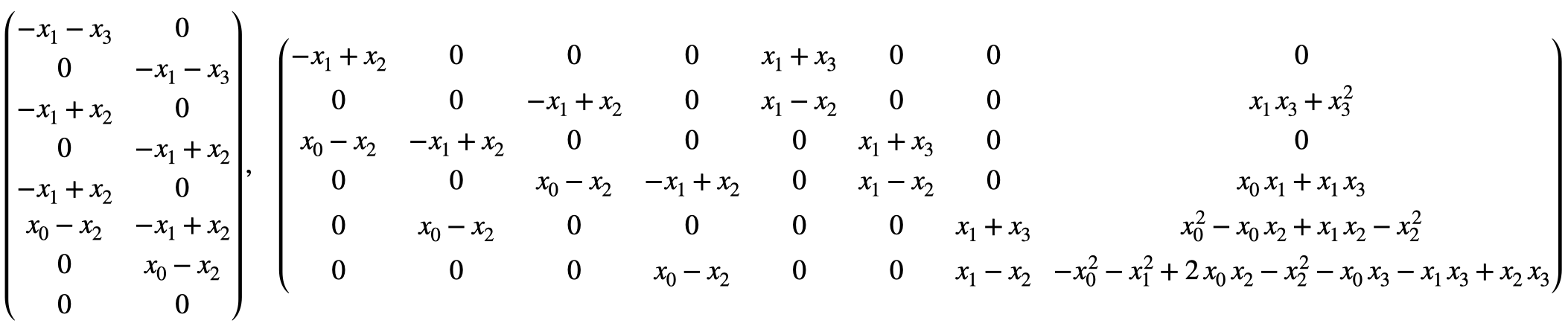}	
\end{center}

By Theorem \ref{thm:quadres}, an implicit equation of $\Sr$ is obtained as the ratio of a 6-minor of $\partial_1$ divided by its corresponding 2-minor of $\partial_2$. Similar expressions are obtained in bidegree $(\mu-1,2)$ for any $\mu\geq 2$ and in bidegree $(2,\nu-1)$ for any $\nu\geq \nu_0=2$. We also observe that in bidegree $(2,2), (3,2)$, $(2,3)$, $(4,2)$ and $(2,4)$ Theorem \ref{thm:quadres} extends the results of Botbol (see Theorem \ref{thm:botbol}) without adding moving quadrics.

It turns out that for this example there is no couple of integers $(\mu,\nu)$ such that the matrix of moving planes and quadrics $\MM_{(\mu,\nu)}$ yields a determinantal representation of $\Sr$. Indeed, we already examined matrices $\MM_{(\mu-1,n-1)}$ and $\MM_{(m-1,\nu-1)}$. Now, if one seeks for determinantal representations of $\Sr$ by direct computations, since $\deg(\Sr)=5$ the other possible bidegrees to examine are $(1,1), (3,0)$ and $(0,3)$. However, in these three cases one never gets a determinantal formula.
\end{ex}

\subsection*{Acknowledgments}
Computations in {\tt Macaulay2} \cite{M2} provided evidence for our work and we thank its authors and contributors. 
We are also grateful to the organizers of the conference ``Ideals, Varieties, Applications'' that took place in Amherst, June 10-14, 2019, and was held in honor of David Cox on the occasion of his retirement. It was a wonderful opportunity to celebrate his influence that we hope this paper illustrates once again.

\end{document}